\documentclass[11pt]{article}

\usepackage{amsfonts}
\newcommand{\R}{\mathbb{R}}

\usepackage{graphicx}
\usepackage{enumerate}
\usepackage{amsmath}
\usepackage{amssymb}
\usepackage{mathtools}
\usepackage{amsthm}
\usepackage{hyperref}
\usepackage{enumerate}
\usepackage{color}

\usepackage{mathrsfs}
\usepackage[left=1in,right=1in, top=1.5in, bottom=1.5in]{geometry}

\usepackage[tiny]{titlesec}

\makeatletter
\def\cleardoublepage{\clearpage\if@twoside \ifodd\c@page\else
\hbox{}
\vspace*{\fill}
\begin{center}
\end{center}
\vspace{\fill}
\thispagestyle{empty}
\newpage
\if@twocolumn\hbox{}\newpage\fi\fi\fi}
\makeatother

\makeatletter
\def\keywords{\xdef\@thefnmark{}\@footnotetext}
\makeatother
\newtheorem{theorem}{Theorem}[section]
\newtheorem{lemma}[theorem]{Lemma}

\newtheorem{proposition}[theorem]{Proposition}
\newtheorem{corollary}[theorem]{Corollary}
\newtheorem{remark}[theorem]{Remark}

\newtheorem{conjecture}[theorem]{Conjecture}

\theoremstyle{definition}
\newtheorem{definition}[theorem]{Definition}

\newtheorem{thmy}{Theorem}
\newenvironment{oldtheorem}{\stepcounter{thm}\begin{thmy}}{\end{thmy}}

\title{\textbf{On a $j$-Santal\'{o} Conjecture}}
\author{Pavlos Kalantzopoulos, Christos Saroglou}
\date{}
\begin{document}
\maketitle
\keywords{Key words and phrases. convex bodies, $j$-Santal\'{o} conjecture, polarity.}
\begin{abstract}
Let $k\geq 2$ be an integer. In the spirit of Kolesnikov-Werner \cite{KW}, for each $j\in\{2,\ldots,k\}$, we conjecture a sharp Santal\'{o} type inequality (we call it $j$-Santal\'{o} conjecture) for many sets (or more generally for many functions), which we are able to confirm in some cases, including the case $j=k$ and the unconditional case. Interestingly, the extremals of this family of inequalities are tuples of the $l_j^n$-ball. Our results also strengthen one of the main results in \cite{KW}, which corresponds to the case $j=2$. All members of the family of our conjectured inequalities can be interpreted as generalizations of the classical Blaschke-Santal\'{o} inequality.
Related, we discuss an analogue of a conjecture due to K. Ball \cite{Ball-conjecture} in the multi-entry setting and establish a connection to the $j$-Santal\'{o} conjecture.
\end{abstract}
\section{Introduction}
\hspace*{1.5em}Throughout this paper we fix an inner product $\langle \cdot ,\cdot\rangle$ in $\R^n$ and an orthonormal basis $\{e_m\}_{m=1}^n$ with respect to this inner product. For a vector $\widetilde{x}=(x_1,\ldots,x_{n-1})\in e_n^\perp$ and a real number $r$, the pair $(\widetilde{x},r)$ will always denote the vector $x_1e_1+\ldots+x_{n-1}e_{n-1}+re_n$. For $p>0$ and $x\in\R^n$, set $\|x\|_p:=\left(\sum_{m=1}^n|\langle x,e_m\rangle|^p\right)^{1/p}$ and $B_p^n:=\{x\in\R^n:\|x\|_p\leq 1\}$. In particular, if $p\geq 1$, $\|\cdot\|_p$ is just the $l_p$-norm in $\R^n$ and $B_p^n$ is its unit ball.  

Given a convex body $K$ (that is, convex compact set with non-empty interior; we refer to \cite{Sc} for a systematic and very complete study on the topic of convex bodies) in $\R^n$, its polar is defined by
$$K^\circ=\{x\in\R^n:\langle x,y\rangle\leq 1,\ \forall y\in K\}.$$
One of the cornerstones in convex geometry is the Blaschke-Santal\'{o} \cite{Bla, Santalo} inequality (see also \cite{M.P., MP2, LZ} for alternative proofs and extensions). We state it in the symmetric case: Let $K$ be a symmetric (i.e. $K=-K$) convex body in $\R^n$. Then,
\begin{equation}\label{classical-santalo}|K||K^\circ|\leq |B_2^n|^2,\end{equation} with equality if and only if $K$ is an ellipsoid. Here, $|\cdot|$ denotes the volume functional.
As a side note, we mention that the exact minimum of the volume product $|K||K^\circ|$ remains unknown for $n\geq 4$ and it is considered to be a major problem in convexity (see \cite{BM, Mi, Ku, Naz, IS}).  

Functional versions of \eqref{classical-santalo} were obtained in \cite{Ball-f-Santalo, AKM, Lehec, FM}.  A result relevant to this note is due to Fradelizi and Meyer \cite{FM}, which we state in the even case for simplicity. Let $\rho:\R\to\R_+$ be a measurable function. If $f,g:\R^n\to\R_+$ are even integrable functions satisfying $f(x)g(y)\leq \rho(\langle x,y\rangle)$, for all $x,y\in\R^n$, then the following inequality holds
$$\int_{\R^n}f(x)\, dx\int_{\R^n}g(y)\, dy\leq \left(\int_{\R^n}\rho(\|u\|_2^2)^{1/2}\, du\right)^2.$$
Very recently, Kolesnikov and Werner \cite{KW} conjectured that, for $\rho$ decreasing, the result of Fradelizi and Meyer holds for several functions, in the following sense.
\begin{conjecture}(Kolesnikov-Werner \cite{KW})\label{conj-K-W-1}
Let $k\geq 2$ be an integer, $\rho:\R\to\R_+$ be a decreasing function and $f_1,\ldots,f_k:\R^n\to\R_+$ be even integrable functions, such that
 \begin{equation}\label{KW-condition1}\prod_{i=1}^kf_i(x_i)\leq \rho\left(\sum_{1\leq i<l\leq k}\langle x_i,x_l\rangle\right),\qquad \forall x_1,\ldots,x_k\in\R^n.\end{equation}Then, it holds
\begin{equation*}
\prod_{i=1}^k\int_{\R^n}f_i(x_i)\,dx_i\leq\left(\int_{\mathbb{R}^n}\rho\left(\frac{k(k-1)}{2}\|u\|^2_2\right)^{1/k}\, du\right)^{k}.
\end{equation*}
\end{conjecture}
They proved the following.
\begin{oldtheorem}\label{KW-Theorem}(Kolesnikov-Werner \cite{KW}) Conjecture \ref{conj-K-W-1} holds if $f_1,\ldots,f_k$ are unconditional (with respect to the same orthonormal basis $\{e_m\}$).
\end{oldtheorem}
One of the goals of this note is to relax the assumption of all $f_1,\ldots,f_k$ being unconditional in Theorem \ref{KW-Theorem} above, to $k-2$ of them being unconditional. 
While polarity in the case $k=2$ is well understood (see \cite{BS, AM}; see also \cite{zoo} for generalizations), it is not clear if there is such a notion for $k>2$. Therefore, we believe it is meaningful to seek for Santal\'{o} type inequalities for sets and for functions under different conditions than \eqref{KW-condition1}.
More precisely, we give the following definition.
\begin{definition}
Let  $\Phi:(\R^n)^k\to\R$ be a function. We say that the sets $K_1,\ldots,K_k\subseteq \R^n$ satisfy \textit{$\Phi$-polarity condition}, if
\begin{equation*}\label{K-Phi-polar}
\Phi(x_1,\ldots,x_k)\leq 1,\qquad \forall x_i\in K_i,\ i=1,\ldots,k.
\end{equation*}
Similarly, we say that the functions $f_1,\ldots,f_k:\R^n\to\R_+$ satisfy $\Phi$-polarity condition with respect to some decreasing function $\rho:\R\to [0,\infty]$, if
\begin{equation*}\label{f-Phi-polar}
\prod_{i=1}^kf_i(x_i)\leq \rho(\Phi(x_1,\ldots,x_k)),\qquad \forall x_i\in\R^n,\ i=1,\ldots,k.
\end{equation*}
\end{definition}
In this note, we will consider a specific family of functions $\Phi$. For integers $1\leq j\leq k$, we set
\begin{equation*}
\mathcal{S}_j(x_1,\ldots,x_k):=\sum_{l=1}^ns_j(x_1(l),\ldots,x_k(l)),\qquad x_1,\ldots,x_k\in\R^n,
\end{equation*}
where $x(l)$ is the $l$'th coordinate (with respect to our fixed basis $\{e_m\}$) of a vector $x\in\R^n$, $l=1,\ldots,n$,  and $s_j$ is the elementary symmetric polynomial of $k$ variables and degree $j$, i.e.
\begin{equation*}
s_j(r_1,\ldots,r_k):=\sum_{1\leq i_1<\ldots< i_j\leq k}r_{i_1}\cdots r_{i_j},\qquad r_1,\ldots,r_k\in \R.
\end{equation*}
Set, also, $$\mathcal{E}_j:=\frac{{\cal S}_j}{{{k}\choose{j}}}.$$Note that for $j\neq 2$ the map $\mathcal{E}_j$ (or ${\cal S}_j$) depends on the basis $\{e_m\}$. However, for $j=2$ this is not the case; one can check that
\begin{equation*}
\mathcal{S}_2(x_1,\ldots,x_k)=\sum_{1\leq i<l\leq k}\langle x_i,x_l\rangle.
\end{equation*}

We conjecture that a Santal\'{o} type inequality holds for symmetric sets, under the assumption of ${\cal E}_j$-polarity condition, $j\ge  2$.
\begin{conjecture}($j$-Santal\'{o} conjecture)\label{S_j Santalo for bodies}
Let $2\leq j\leq k$ be two integers, where $k\geq 2$. If $K_1,\ldots,K_k$ are symmetric convex bodies  in $\R^n$, satisfying $\mathcal{E}_j$-polarity condition, then it holds
\begin{equation}\label{santalo-for-many}
\prod_{i=1}^k|K_i|\leq |B^n_j|^k.
\end{equation}
\end{conjecture}
\begin{remark}\label{Fradelizi}
	We exclude the case $j=1$, because the quantity $|K_1|\ldots|K_k|$ can be unbounded for bodies $K_1,\ldots,K_k$ satisfying $\mathcal{E}_1$-polarity condition. This can be seen by taking all $K_i$ to be the symmetric slab $\{x\in\R^n:|x_1+\ldots+x_n|\leq1\}$.
\end{remark}
We, also, formulate the functional version of Conjecture \ref{S_j Santalo for bodies}.
\begin{conjecture}(Functional $j$-Santal\'{o} conjecture)\label{jvnljfnv}
Let $2\leq j\leq k$ be two integers, where $k\geq 2$. If $f_1,\ldots,f_k:\R^n\to\R_+$ are even integrable functions, satisfying $\mathcal{S}_j$-polarity condition with respect to some decreasing function $\rho:\R\to[0,\infty]$, then it holds
\begin{equation}\label{S_j-ineq}
\prod_{i=1}^k\int_{\R^n}f_i(x_i)\,dx_i\leq\left(\int_{\mathbb{R}^n}\rho\left({{k}\choose{j}}\|u\|^j_j\right)^{1/k}\, du\right)^{k}.
\end{equation}
\end{conjecture}
We remark that, unlike the case $j=1$, $|K_1|\ldots|K_k|$ is always bounded, if $j\geq 2$ and $K_1,\ldots,K_k$ satisfy $\mathcal{E}_j$-polarity condition. This follows from Remark \ref{boundremark} below and similar statement also holds for the functional version (see Remark \ref{boundremark2}). 

Clearly, for $j=2$, Conjecture \ref{jvnljfnv} is just Conjecture \ref{conj-K-W-1}.
As it expected, the functional $j$-Santal\'{o} Conjecture \ref{jvnljfnv}  implies the $j$-Santal\'{o} Conjecture \ref{S_j Santalo for bodies} for sets. To see this, let $K_1,\ldots,K_k$ be symmetric convex bodies satisfying $\mathcal{E}_j$-polarity condition. Then, setting  
\begin{equation}\label{observation}
f_i:=1_{K_i} , \  i=1,\ldots,k
\qquad \textnormal{and}\qquad\rho(t):=\begin{cases}
+\infty,&t<0\\
1_{[0,1]}\left({{k}\choose{j}}^{-1}t\right)&t\geq 0\end{cases},
\end{equation}
one can check that the functions $f_1,\ldots,f_k$ satisfy $\mathcal{S}_j$-polarity condition with respect to $\rho$, where with $1_A$ denotes the indicator function of a set $A$. Therefore, inequality \eqref{S_j-ineq} implies,
\begin{equation*}
\prod_{i=1}^k|K_i|\leq\left(\int_{\R^n}\rho\left({{k}\choose{j}}\|u\|^j_j\right)^{\frac{1}{k}}\,du\right)^k=\left(\int_{\R^n}1_{[0,1]}\left(\|u\|^j_j\right)^{\frac{1}{k}}\,du\right)^k=\left(\int_{B^n_j}1\,du\right)^k=|B^n_j|^k.
\end{equation*}
It turns out, however, that the two conjectures are actually equivalent (see Section \ref{Bodies.iff.functions}).
Let us state our main results.
\begin{theorem}\label{phisantalo}
Conjecture \ref{S_j Santalo for bodies} holds in the following cases:
\begin{enumerate}[(i)]
\item $K_1,\ldots, K_k$ are unconditional convex bodies.
\item $j=k$.
\item $j$ is even and $K_3,\ldots,K_k$ are unconditional convex bodies.
\end{enumerate}
Moreover, in all three cases, \eqref{santalo-for-many} is sharp for $K_1=K_2=\ldots=K_k=B_j^n$.
\end{theorem}
\begin{theorem}\label{KWconj}
Conjecture \ref{jvnljfnv} holds in the following cases:
\begin{enumerate}[(i)]
\item $f_1,\ldots, f_k$ are unconditional functions.
\item $j=k$.
\item $j$ is even and $f_3,\ldots,f_k$ are unconditional functions.  
\end{enumerate}
\end{theorem}
Notice that, by \eqref{observation}, \eqref{S_j-ineq} is also sharp for some specific choice of $\rho$.
As mentioned earlier, Theorem \ref{KWconj} (case (iii), $j=2$) slightly extends Theorem \ref{KW-Theorem}. Moreover, Theorem \ref{phisantalo} (resp. \ref{KWconj}) for $j=k$ can be viewed as a generalization of the classical Blasckhe-Santal\'{o} inequality in the setting of many sets (resp. many functions). Although, due to \eqref{observation}, Theorem \ref{KWconj} implies immediately Theorem \ref{phisantalo}, for parts (ii) and (iii), our approach requires  to first establish Theorem \ref{phisantalo} and then deduce Theorem \ref{KWconj} from it and the equivalence between Conjectures \ref{S_j Santalo for bodies} and \ref{jvnljfnv} mentioned above. 
The unconditional case is established in Section \ref{P.L.-unc-section}. Part (ii) for $j=k$ and part (iii) are established in Section \ref{Steiner-symmetrization}, using symmetrization arguments, similar to the ones used by Meyer-Pajor \cite{M.P.}. Finally, the equivalence of Conjectures \ref{S_j Santalo for bodies} and \ref{jvnljfnv} is demonstrated in Section \ref{Bodies.iff.functions}.

K. Ball \cite{Ball-f-Santalo}  \cite{Ball-conjecture}, defined the following $SL(n)$-invariant quantity
\begin{equation*}
B(K):=\int_{K}\int_{K^o}\langle x,y \rangle^2\,dx\,dy.  
\end{equation*}
Using the multiplicative version of Pr\'ekopa-Leindler, (Theorem \ref{Prekopa-Leindlre} below), he obtained the following.
\begin{oldtheorem}\label{Ball-Theorem}
If $K$ is an unconditional convex body in $\R^n$, then $$B(K)\leq B(B_2^n).$$
\end{oldtheorem}
Ball conjectured the following.
\begin{conjecture}\label{Ball-conjec}
Theorem \ref{Ball-Theorem} holds true for arbitrary symmetric convex bodies.
\end{conjecture}
 We refer to Conjecture \ref{Ball-conjec} as Ball's conjecture. We should remark that Ball has shown that Conjecture \ref{Ball-conjec}, if true, implies the Blaschke-Santal\'{o} inequality
\eqref{classical-santalo} (in the sense that given the validity of Conjecture \eqref{Ball-conjec}, one can prove \eqref{classical-santalo} within a few lines). In Section \ref{multi-Ball-section}, we formulate the analogue of Ball's conjecture for many sets that satisfy ${\cal E}_j$-polarity condition (resp. many functions that satisfy ${
\cal S}_j$-polarity condition). We prove it in the unconditional case and we show that this implies Conjectures \ref{S_j Santalo for bodies} and \ref{jvnljfnv}.
\section{ 
The unconditional case}\label{P.L.-unc-section}
\hspace*{1.5em}A function $f:\R^n\to\R$ is said to be unconditional, if $f(\delta_1 x_1,\ldots,\delta_n x_n)=f(x_1,\ldots,x_n)$, for any $\delta_1,\ldots,\delta_n\in\{-1,1\}$ and any $(x_1,\ldots,x_n)\in\R^n$. A subset $A$ in $\R^n$ is unconditional if its indicator function $1_A$ is unconditional.
In this section, we establish Conjectures \ref{S_j Santalo for bodies} and \ref{jvnljfnv} for sets contained in an orthant (resp. functions supported in an orthant). Since we wish to obtain slightly more general results, we need to modify the definition of the functions ${\cal S}_j$ and ${\cal E}_j$ introduced previously. Namely, for any two integers $1\leq j\leq k$ and any positive real number $p>0$, set
\begin{equation*}\label{s_jp}
s_{j,p}(r_1,\ldots,r_k)=\sum_{1\leq i_1<\ldots<i_j\leq k}|r_{i_1}|^p\ldots |r_{i_j}|^p,\qquad r_1,\ldots,r_k\in\R,
\end{equation*}

\begin{equation*}
S_{j,p}(x_1,\ldots,x_k):=\sum_{l=1}^ns_{j,p}(x_1(l),\ldots,x_k(l)),\qquad x_1,\ldots,x_k\in\R^n
\end{equation*}
and $$\mathcal{E}_{j,p}:=\frac{S_{j,p}}{{{k}\choose{j}}}.$$ Recall that $x_i(l):=\langle x_i,e_l\rangle$, $i=1,\ldots,k$, $l=1,\ldots,n$.

We refer to Borell \cite{Borell}, Ball \cite{Ball-conjecture} \cite{Ball-last} , Uhrin \cite{Uhrin} for the following version of the classical Pr\'ekopa-Leindler inequality (see also \cite{Pr}, \cite{le}).
\begin{oldtheorem}(1-dimensional multiplicative Pr\'ekopa-Leindler inequality)\label{Prekopa-Leindlre}
If some integrable functions  $h,h_i:\R_+\to\R_+$, $i=1,\ldots,k$, satisfy
\begin{equation*}\label{PL.condition}
\prod_{i=1}^kh_i(t_i)^{\frac{1}{k}}\leq h\left(\prod_{i=1}^kt_i^{\frac{1}{k}}\right),\qquad \forall t_i>0, \ i=1\ldots,k,
\end{equation*}
then it holds \begin{equation*}\label{P.L.}
\prod_{i=1}^k\left(\int_{\R_+}h_i(t_i)\,dt_i\right)^{\frac{1}{k}}\leq \int_{\R_+}h(t)\,dt.
\end{equation*}
\end{oldtheorem}

The result below slightly generalizes Theorem \ref{KW-Theorem} and will follow by a more or less standard application of the Pr\'ekopa-Leindler inequality (Theorem \ref{Prekopa-Leindlre}), which uses the inductive argument of K. Ball \cite{Ball-f-Santalo, Ball-conjecture}.
\begin{proposition}\label{Sjp-BALL-orthants}
Let $p>0$, $q>-1$ be real numbers and $1\leq j\leq k$ be two integers, where $k\geq2$. For any integrable functions $f_i:\R^n_+\to\R_+$, $i=1,\ldots,k$, satisfying $S_{j,p}$-polarity condition with respect to some decreasing function $\rho:\R\to[0,\infty]$,
and any $m\in\{1,\ldots,n\}$, it holds
\begin{equation}\label{BALL-unconditional-S,j,p}
\prod_{i=1}^k\int_{\R^n_+}|\langle x_i,e_m\rangle|^qf_i(x_i)\,dx_i\leq\left(\int_{\R^n_+}|\langle u,e_1\rangle|^q\rho\left({{k}\choose{j}}\|u\|_{jp}^{jp}\right)^{\frac{1}{k}}\,du\right)^k.
\end{equation}
\end{proposition}
\begin{proof}
We may assume that $m=n$. We will prove Proposition \ref{Sjp-BALL-orthants} by induction in the dimension $n$. Since we want to deduce the base case $n=1$ simultaneously with the inductive step, it is useful to make some conventions: For a function $\varphi:\R_+\to\R_+$, we set $\varphi(0,r):=\varphi(r)$, $\int_{\R_+^0}\varphi(x)\, dx:=\varphi(0)$ and 
$\int_{\R^0_+}\varphi(x,r)\, dx:=\varphi(0,r)=\varphi(r)$, $r\geq 0$. 
Assume that \eqref{BALL-unconditional-S,j,p} holds for the non-negative integer $n-1$, where $n\geq 2$.  
In the inductive step (resp. the case $n=1$), notice that for $(\widetilde{x}_i,r_i)\in\R_+^n$ (resp. $r_i\in\R_+$), $i=1,\ldots,k$, the $S_{j,p}$-polarity condition together with  Maclaurin's inequality (stating that ${\cal E}_{j_1}^{1/j_1}\geq {\cal E}_{j_2}^{1/j_2}$, if $j_1\leq j_2$) and the monotonicity of $\rho$ imply
\begin{equation*}
\prod_{i=1}^kf_i(\tilde{x}_i,r_i)\leq \rho\left(s_{j,p}(r_1,\ldots,r_k)+S_{j,p}(\tilde{x}_1,\ldots,\tilde{x}_k)\right)
\leq\rho\left({{k}\choose{j}}(r_1\ldots r_k)^{\frac{jp}{k}}+S_{j,p}(\tilde{x}_1,\ldots,\tilde{x}_k)\right),
\end{equation*}where $\widetilde{x}_1=\ldots=\widetilde{x}_k:=0$, if $n=1$.
Multiplying by $\prod_{i=1}^kr_i^q$ we get
\begin{equation}\label{pro-ind}
\prod_{i=1}^kr_i^qf_i(\tilde{x_i},r_i)\leq\prod_{i=1}^kr_i^q\cdot\rho\left({{k}\choose{j}}\prod_{i=1}^kr_i^{\frac{jp}{k}}+S_{j,p}(\tilde{x}_1,\ldots,\tilde{x}_k)\right).
\end{equation}
For fixed $r_1,\ldots,r_k> 0$, set
\begin{equation*}
\widetilde{\rho}(t):=\prod_{i=1}^kr_i^q\cdot\rho\left({{k}\choose{j}}\prod_{i=1}^kr_i^{\frac{jp}{k}}+t\right),\qquad t\geq 0.
\end{equation*}
Applying the inductive hypothesis for $q=0$ to \eqref{pro-ind} if $n\geq 2$ or the conventions made above if $n=1$, we obtain
\begin{eqnarray}\label{ind.step}
\prod_{i=1}^k\int_{\R^{n-1}_+}r_i^qf_i(\tilde{x}_i,r_i)\,d\tilde{x}_i&\leq& \left(\int_{\R^{n-1}_+}\widetilde{\rho}\left({{k}\choose{j}}\|\tilde{u}\|_{jp}^{jp}\right)^{\frac{1}{k}}\,d\tilde{u}\right)^k\nonumber\\
&=&\left( \left(\prod_{i=1}^kr_i^{\frac{1}{k}}\right)^q\int_{\R^{n-1}_+}
\rho\left({{k}\choose{j}}\prod_{i=1}^kr_i^{\frac{jp}{k}}+{{k}\choose{j}}\|\tilde{u}\|^{jp}_{jp}\right)^{\frac{1}{k}} \,d\tilde{u}\right)^k.
\end{eqnarray}
For $t,r_i>0$, $i=1,\ldots,k$, set $$h_i(r_i):=\int_{\R^{n-1}_+}r_i^qf_i(\tilde{x}_i,r_i)\,d\tilde{x}_i
\qquad\textnormal{and}\qquad h(t):=t^q\int_{\R^{n-1}_+}\rho\left({{k}\choose{j}}(t^{jp}+\|\tilde{u}\|^{jp}_{jp})\right)^{\frac{1}{k}}\,d\tilde{u}.
$$
Then, by \eqref{ind.step}, the functions $h,h_1,\ldots,h_k$ satisfy the assumption of Theorem \ref{Prekopa-Leindlre}, hence
\begin{equation*}\label{prekopa-leilder}
\prod_{i=1}^k\int_{\R_+}\int_{\R^{n-1}_+}r_i^qf_i(\tilde{x_i},r_i)\,d\tilde{x}_i\,dr_i\leq \left( \int_{\R_+}t^q\int_{\R^{n-1}_+}\rho\left({{k}\choose{j}}\big(t^{jp}+\|\tilde{u}\|^{jp}_{jp}\big)\right)^{\frac{1}{k}}\,d\tilde{u}\,dt \right)^k.
\end{equation*}
The assertion follows by Fubini's Theorem.
\end{proof}
Setting $q=0$ and $p=1$ to Proposition \ref{Sjp-BALL-orthants}, it follows immediately that Conjecture \ref{jvnljfnv} holds true for functions $f_1,\ldots,f_k$ supported in $\R^n_+$. Moreover, by \eqref{observation}, Conjecture \ref{S_j Santalo for bodies} holds  for convex bodies $K_1,\ldots, K_k$, contained in $\R^n_+$. In particular, case (i) of Theorems \ref{phisantalo} and \ref{KWconj} follows from the previous discussion.  
\begin{corollary}\label{main.results.for.unco.}
Conjectures \ref{S_j Santalo for bodies} and \ref{jvnljfnv} hold in the unconditional setting.
\end{corollary}
We also have the following.
\begin{corollary}\label{S_jpbcsd}
Proposition \ref{Sjp-BALL-orthants} holds if $\R^n_+$ is replaced by $\R^n$.
\end{corollary}
\begin{proof}
Let $O_i$, $i=1,\ldots,2^n$ be an enumeration of all orthants. Then,
\begin{equation*}
\prod_{i=1}^k\int_{\R^n}|\langle x_i,e_m\rangle|^qf_i(x_i)\,dx_i=\sum_{l_1,\ldots,l_k=1}^{2^n}\prod_{i=1}^k\int_{O_{l_{i}}}|\langle x_{i},e_m\rangle|^qf_{{i}}(x_{i})\,dx_{i}.
\end{equation*}
Therefore, it suffices to prove that if $l_1,\ldots,l_k\in\{1,\ldots,2^n\}$, then
\begin{equation*}
\prod_{i=1}^k\int_{O_{l_i}}|\langle x_{i},e_m\rangle|^qf_{{i}}(x_{i})\,dx_{i}\leq \left(\int_{\R^n_+}|\langle u,e_m\rangle|^q\rho\left({{k}\choose{j}}\|u\|_{jp}^{jp}\right)^{\frac{1}{k}}\,du\right)^k.
\end{equation*}
Let $\tilde{f}_{{i}}:\R^n_+\to\R_+$ be the function defined by $\tilde{f}_{{i}}(|x(1)|,\ldots,|x(n)|)=f_{{i}}(x(1),\ldots,x(n))$, $i=1,\ldots,k$. Notice that, for any $x_{1},\ldots,x_{k}\in\R^n_+$ it holds
\begin{equation*}
\prod_{i=1}^k\tilde{f}_{{i}}(x_i)\leq \rho(S_{j,p}(x_1,\ldots,x_k))
\end{equation*}
and also
\begin{equation*}
\prod_{i=1}^k\int_{\R^n_+}|\langle x_{i},e_m\rangle|^q\tilde{f}_{{i}}(x_{i})\,dx_{i}=\prod_{i=1}^k\int_{O_{l_i}}|\langle x_{i},e_m\rangle|^qf_{{i}}(x_{i})\,dx_{i}.
\end{equation*}
The desired inequality follows by Proposition \ref{Sjp-BALL-orthants}.
\end{proof}
Setting $q=0$ to Corollary \ref{S_jpbcsd} and using \eqref{observation}, we obtain the following.
\begin{corollary}\label{S_jpbcsd1}
Let $p>0$ and $1\leq j\leq k$ be two integers, where $k\geq 2$. Then, for any integrable functions  $f_i:\R^n\to\R$, $i=1,\ldots,k$, satisfying $S_{j,p}$-polarity condition with respect to some decreasing function $\rho:\R\to[0,\infty]$,
it holds
\begin{equation*}
\prod_{i=1}^k\int_{\R^n}f_i(x_i)\,dx_i\leq\left(\int_{\R^n}\rho\left({{k}\choose{j}}\|u\|_{jp}^{jp}\right)^{\frac{1}{k}}\,du\right)^k.
\end{equation*}
Moreover, if $K_1,\ldots,K_k$ are any convex bodies in $\R^n$ satisfying $\mathcal{E}_{j,p}$-polarity condition, one has
\begin{equation*}
\prod_{i=1}^k|K_i|\leq |B^n_{jp}|^k.
\end{equation*}
\end{corollary}
\begin{remark}\label{tapa1}
Kolesnikov and Werner \cite[Proposition 5.5.]{KW} established a related result, stating  the following. If $K_1,\ldots,K_k$ are unconditional convex bodies satisfying  
\begin{equation}\label{KW-PAPARIA}
\prod_{i=1}^k{r_{K_i}}(x_i)\leq\left( \sum_{l=1}^n|x_1(l)|^{\frac{2}{k}}\cdots|x_k(l)|^{\frac{2}{k}} \right)^{-\frac{k}{2}},\qquad \forall x_i\in S^{n-1},  \ i=1,\ldots,k
\end{equation}
then $|K_1|\cdots|K_k|\leq |B^n_2|^k$, where  $r_{K_i}(u):=\sup\{t>0:tu\in K_i\}$, $u\in\R^n$, denotes the radial function of $K_i$. One can notice that, since $r_{K_i}(\cdot)=\|\cdot\|_{K_i}^{-1}$, \eqref{KW-PAPARIA} is equivalent to,
\begin{equation*}\label{S,k,2/k}
\sum_{l=1}^n \Big|\frac{x_1(l)}{\|x_1\|_{K_1}}\Big|^{\frac{2}{k}}\cdots\Big|\frac{x_k(l)}{\|x_k\|_{K_k}}\Big|^{\frac{2}{k}}\leq 1,\qquad \forall x_i\in \R^n,  \ i=1,\ldots,k,
\end{equation*}
which can be written as
\begin{equation*}\label{S,k,2/k1}
\sum_{l=1}^n |x_1(l)|^{\frac{2}{k}}\cdots|x_k(l)|^{\frac{2}{k}}\leq 1 \qquad \forall x_i\in K_i,  \ i=1,\ldots,k.
\end{equation*}
Thus, by Corollary \ref{S_jpbcsd1} for  $j=k$ and $p=2/k$, $|K_1|\cdots|K_k|\leq |B^n_2|^k$ holds for any convex bodies $K_1,\ldots,K_k$ (not necessarily symmetric) that satisfy condition \eqref{KW-PAPARIA}.
\end{remark}
\section{Symmetrization}\label{Steiner-symmetrization}
\hspace*{1.5em}This section is devoted to completing the proof of Theorem \ref{phisantalo}. Our proof based on a modification of a symmetrization technique used in \cite{M.P.} and on the following observation. If $j$ is even and $r_1,\ldots,r_k\in\R$, then
\begin{equation}\label{sdcsdcsdc}
s_j(r_1,-r_2,\ldots,-r_k)=s_j(-r_1,r_2,\ldots,r_k),
\end{equation}
while\begin{equation}\label{sdcsdcsdc1}
s_k(r_1,-r_2,r_3,\ldots,r_k)=s_k(-r_1,r_2,r_3,\ldots,r_k).
\end{equation}
The proof will follow easily from the next lemma.
\begin{lemma}\label{cknsldkcs}
Let $2\leq j\leq k$ and $K_1,\ldots,K_k$ symmetric convex bodies satisfying $\mathcal{E}_j$-polarity condition. Assume that one of the following holds
\begin{enumerate}[(i)]
\item $j=k$.
\item $j$ is even and $K_3,\ldots,K_k$ are unconditional.
\end{enumerate}
Then there exist  $U_1,\ldots,U_k$ unconditional convex bodies satisfying $\mathcal{E}_j$-polarity condition, such that
\begin{equation*}
\prod_{i=1}^k|K_i|\leq \prod_{i=1}^k|U_i|.
\end{equation*}
\end{lemma}
\begin{proof}
For a set $A\subseteq\R^n$ and a number $r\in\R$, set
\begin{equation*}
A(r):=\{\tilde{x}\in e_n^{\perp}:(\tilde{x},r)\in A\}.
\end{equation*}
The Steiner symmetrization of a convex body $K$ with repsect to $e_n^{\perp}$ is given by
\begin{equation*}
st_{e_n^{\perp}}(K)=\Big\{\big(\tilde{x},\frac{r-r'}{2}\big)\in \R^n :\tilde{x}\in P_{e_n^{\perp}}(K), \   \text{and} \  (\tilde{x},r),(\tilde{x},r')\in K\Big\},
\end{equation*}where $P_{e_n^{\perp}}(K)$ denotes the orthogonal projection of $K$ onto the subspace $e_n^\perp$.

For symmetric convex bodies $K_1,K_3,\ldots,K_k$, set
\begin{equation*}\label{polar-K_2}
(K_1,K_3,\ldots,K_{k})^o_j:=\Big\{x_2\in\R^n:{\cal S}_j(x_1,x_2,\ldots,x_k)\leq {{k}\choose{j}}, \ \text{for all} \  x_i\in K_i  \ \text{with} \  i\neq 2\Big\}.
\end{equation*}
This is a just generalization of the notion of the polar set in the case $j=k=2$. Clearly, the set $(K_1,K_3,\ldots,K_{k})^o_j$ is a symmetric convex body and, furthermore, if the $K_i$ are all unconditional then $(K_1,K_3,\ldots,K_{k})^o_j$ is also unconditional.  Notice also that $(K_1,K_3,\ldots,K_{k})^o_j$ is the largest symmetric convex body, such that the sets   $K_1,(K_1,K_2,\ldots,K_{k})^o_j,K_3,\ldots,K_{k},$ satisfy ${\cal E}_j$-polarity condition.
We will prove both assertions of Lemma \ref{cknsldkcs} simultaneously by Steiner symmetrization. We may clearly assume that $K_2=(K_1,K_3,\ldots,K_k)^o_j$. We set $K_2'=(st_{{e_n}^{\perp}}K_1,K_3,\ldots,K_k)^o_j$. We will show that, for $r\geq0$, it holds  
\begin{equation}\label{inclusion}
\frac{K_2(r)+K_2(-r)}{2}\subseteq K_2'(r).
\end{equation}
Let $\tilde{x}_2\in K_2(r)$ and $\tilde{x}_2'\in K_2(-r)$. Then, for all $(\tilde{x}_i,r_i)\in K_i$, $i=3,\ldots,k,$ and for all $(\tilde{x}_1,r_1),(\tilde{x}_1,r_1')\in K_1$, it holds
\begin{equation}\label{cskjdncs}
{\cal S}_j((\tilde{x}_1,r_1),(\tilde{x}_2,r),(\tilde{x}_3,r_3),\ldots,(\tilde{x}_k,r_k))\leq {{k}\choose{j}}
\end{equation}
and
\begin{equation}\label{jndscs}
{\cal S}_j((\tilde{x}_1,r_1'),(\tilde{x}_2',-r),(\tilde{x}_3,r_3),\ldots,(\tilde{x}_k,r_k))\leq {{k}\choose{j}}.
\end{equation}
When $j$ is even and $K_3,\ldots,K_k$ are unconditional, \eqref{jndscs}  can be equivalently written as
\begin{equation}\label{jndscs112312}
{\cal S}_j((\tilde{x}_1,r_1'),(\tilde{x}_2',-r),(\tilde{x}_3,-r_3),\ldots,(\tilde{x}_k,-r_k))\leq {{k}\choose{j}}.
\end{equation}
Combining \eqref{sdcsdcsdc1}, \eqref{jndscs} and \eqref{sdcsdcsdc}, \eqref{jndscs112312}, we obtain (for both assertions of the Lemma)
\begin{equation}\label{jndscs11}
{\cal S}_j((\tilde{x}_1,-r_1'),(\tilde{x}_2',r),(\tilde{x}_3,r_3),\ldots,(\tilde{x}_k,r_k))\leq {{k}\choose{j}},
\end{equation}
for all $(\tilde{x}_i,r_i)\in K_i$, $i=3,\ldots,k$ and for all $(\tilde{x}_1,r_1')\in K_1$. Averaging \eqref{cskjdncs} and \eqref{jndscs11}, and since ${\cal S}_j$ is affine with respect to each argument, we conclude that
\begin{equation*}
{\cal S}_j\Big(\big(\tilde{x}_1,\frac{r_1-r_1'}{2}\big),\big(\frac{\tilde{x}_2+\tilde{x}_2'}{2},r\big),(\tilde{x}_3,r_3),\ldots,(\tilde{x}_k,r_k)\Big)\leq {{k}\choose{j}},
\end{equation*}
for all $(\tilde{x}_i,r_i)\in K_i$, $i=3,\ldots,k$ and for all $(\tilde{x}_1,r_1),(\tilde{x}_1,r_1')\in K_1$. This shows that $\frac{\tilde{x}_2+\tilde{x}_2'}{2}\in K_2'(r)$, which establishes \eqref{inclusion}.
Inclusion \eqref{inclusion} together with the Brunn-Minkowski inequality and Fubini's Theorem show
that \begin{equation*}
|K_1|=2\int_{0}^{\infty}|K_2(r)|\,dr\leq2\int_{0}^{\infty}|K_2'(r)|\,dr=|K_2'|.
\end{equation*}
Applying the same argument successively with respect to $e_{n-1},\ldots,e_1$, we arrive at an unconditional convex body $U_1$, such that $|U_1|=|K_1|$, the tuple $U_1,\bar{K}_2:=(U_1,K_3,\ldots,K_k)^o_j,K_3,\ldots,K_k$ satisfies ${\cal E}_j$-polarity condition and $|K_2|\leq|\bar{K}_2|$. This can be done for both cases (i) and (ii) of the lemma. 

Recall that if $K_3,\ldots,K_k$ are unconditional then $\bar{K}_2$ is also unconditional and the proof of (i) is complete. 

In the case $j=k$, we repeat the same argument to the new tuple $(U_1,\bar{K}_2,K_3,\ldots,K_k)$ with respect to the pair $(\bar{K}_2, K_3)$. Thus, we are able to replace $\bar{K}_2$ by an unconditional convex body $U_2$ and $K_3$ by a symmetric convex body $\bar{K}_3$, such that $|U_2|=|\bar{K}_2|$, $|\bar{K}_3|\geq |K_3|$, while the tuple $(U_1,U_2,\bar{K}_3,K_4,\ldots,K_k)$ also satisfies ${\cal E}_j$-polarity condition. We continue the same process until we replace all $K_1,\ldots,K_{k-1}$ by unconditional convex bodies $U_1,\ldots,U_{k-1}$ without decreasing the volume product of the $K_i$'s. We conclude the proof by the fact that $U_k:=(U_1,\ldots,U_{k-1})^\circ_k$ is also unconditional.
\end{proof}
\begin{proof}[Proof of Theorem \ref{phisantalo}]
Inequality \eqref{santalo-for-many} in all cases follows from Lemma \ref{cknsldkcs} together with Corollary \ref{main.results.for.unco.}.
It remains to verify that  \eqref{santalo-for-many} is sharp for $l_j$-balls. In other words we need to prove that if  $K_1=\ldots=K_k=B^n_j$, then $K_1,\ldots,K_k$ satisfy $\mathcal{E}_j$-polarity condition . But this is a simple application of the arithmetic-geometric mean inequality: For any $x_1,\ldots,x_k\in B_j^n$, it holds
\begin{eqnarray*}
{\cal S}_j(x_1,\ldots,x_k)&=&\sum_{l=1}^n\sum_{1\leq i_1<\ldots<i_j\leq k}x_{i_1}(l)\cdots x_{i_j}(l)\\
&\leq& \frac{1}{j}\sum_{l=1}^n\sum_{1\leq i_1<\ldots<i_j\leq k}\big(|x_{i_1}(l)|^j+\ldots+|x_{i_j}(l)|^j\big)\\
&=&\frac{(k-1)!}{(k-j)!j!}\sum_{l=1}^n\big(|x_1(l)|^j+\ldots+|x_k(l)|^j\big)\\
&=&\frac{(k-1)!}{(k-j)!j!}\sum_{i=1}^k\|x_i\|_j^j\leq \frac{(k-1)!}{(k-j)!j!}k={{k}\choose{j}}.
\end{eqnarray*}
\end{proof}
\begin{remark}\label{boundremark}
	The case $j=k$ of Theorem \ref{phisantalo} implies that for any $2\leq j\leq k$ there exists a positive constant $a_{n,j,k}$ that depends only on $n,j,k$, such that for any symmetric convex bodies $K_1,\ldots,K_k$ in $\R^n$, satisfying $\mathcal{E}_j$-polarity condition, it holds
	\begin{equation}\label{boundjj}
	\prod_{i=1}^k|K_i|\leq a_{n,j,k}|B^n_j|^k,
	\end{equation} 
	where one can take $a_{n,j,k}={{k}\choose{j}}^{nk/j}$.
\end{remark}
To see this, notice that $S_j(x_1,\ldots,x_j,0,\ldots,0)\leq {{k}\choose{j}}$, for any $x_i\in K_i$, $i=1,\ldots,j$. This implies easily that the tuple ${{k}\choose{j}}^{-1/j}K_1,\ldots,{{k}\choose{j}}^{-1/j}K_j$ satisfies $\mathcal{E}_j$-polarity condition, thus by Theorem \ref{phisantalo} (ii) one has 
\begin{equation*}
\prod_{i=1}^j|K_i|\leq {{k}\choose{j}}^{n}|B^n_j|^j.
\end{equation*}
The same argument shows that a similar inequality holds if $K_1,\ldots,K_j$ are replaced by $K_{i_1},\ldots,K_{i_j}$, with $1\leq i_1<\ldots<i_j\leq k$. Multiplying all these inequalities together, yields \eqref{boundjj}.

\section{Equivalence between Conjectures \ref{S_j Santalo for bodies} and \ref{jvnljfnv}}\label{Bodies.iff.functions}
\hspace*{1.5em}In this section we prove cases (ii) and (iii) of Theorem \ref{KWconj}. This is done by establishing the equivalence between Conjectures \ref{S_j Santalo for bodies} and \ref{jvnljfnv} (actually, a slightly more general result), mentioned in the Introduction. Let us first introduce some notation. Let $G$ be a subgroup of the orthogonal group $O(n)$ in $\R^n$. We set
\begin{equation*}
\mathcal{S}(G):=\{S\subseteq\R^n:gS=S, \ \forall  \ g\in G\} \qquad \textnormal{and}\qquad \mathcal{F}(G):=\{f:\R^n\to \R:f\circ g=f, \ \forall  \ g\in G\}.
\end{equation*}
\begin{proposition}\label{vjndsf}
Let $\mu$ be an $a$-homogeneous Borel measure in $\R^n$ for some $a>0$, $k$ be a positive integer, $j\in\{2,\ldots,k\}$ and $G_1,\ldots,G_k$ be subgroups of $O(n)$. The following statements are equivalent.
\begin{enumerate}[i)]
    \item For any $k$-tuple of symmetric convex bodies $(K_1,\dots,K_k)\in\mathcal{S}(G_1)\times\cdots\times \mathcal{S}(G_k)$, satisfying ${\cal E}_j$-polarity condition, it holds
    \begin{equation*}\label{eq-prop-equiv-1'}
    \prod_{i=1}^k \mu(K_i)\leq \mu(B_j^n)^k.    
    \end{equation*}
    \item For any $k$-tuple of even non-negative measurable functions $(f_1,\ldots,f_k)\in\mathcal{F}(G_1)\times\cdots\times \mathcal{F}(G_k)$, satisfying ${\cal S}_j$-polarity condition with respect to some decreasing function $\rho$, it holds
    \begin{equation}\label{eq-prop-equiv-2'}
    \prod_{i=1}^k\int_{\R^n}f_i(x_i)\, d\mu(x_i)\leq \left(\int_{\R^n}\rho\left({{k}\choose{j}}\|u\|_j^j\right)^{1/k}\, d\mu(u)\right)^k.    
    \end{equation}
\end{enumerate}
\end{proposition}
The fact that Conjectures \ref{S_j Santalo for bodies} and \ref{jvnljfnv} are equivalent follows immediately from Proposition \ref{vjndsf}, if we take $\mu$ to be the Lebesgue measure. For the proof we will need the following lemma (which is well known in the classical case $j=k=2$).
\begin{lemma}\label{kjsdncs}
Let $1\leq j\leq k$ and $A_1,\ldots,A_k$ be subsets of $\R^n$. If $A_1,\ldots,A_k$ satisfy $\mathcal{E}_j$-polarity condition, then $\textnormal{conv} (A_1),\ldots,\textnormal{conv} (A_k)$ also satisfy $\mathcal{E}_j$-polarity condition.
\end{lemma}
\begin{proof}
Clearly, it suffices to prove that $\textnormal{conv} (A_1),A_2,\ldots,A_k$ satisfy $\mathcal{E}_j$-polarity condition. This follows from the observation that, if  $\lambda
_1,\ldots,\lambda_r\geq 0$ are real numbers that sum to 1 and if $x_2,\ldots,x_k\in\R^n$ and  $y_1,\ldots,y_r\in\R^n$, then
\begin{equation*}
\mathcal{E}_j\Big(\sum_{m=1}^r\lambda_my_m,x_2,\ldots,x_k\Big)=\sum_{m=1}^r\lambda_m\mathcal{E}_j(y_m,x_2,\ldots,x_k).
\end{equation*}
\end{proof}
\begin{proof}[Proof of Proposition \ref{vjndsf}]
The fact  that (ii) implies (i), follows immediately from \eqref{observation}.

For the other direction, assume that (i) holds for all bodies $K_i\in \mathcal{S}(G_i)$, $i=1,\ldots,k$. Let $(f_1,\ldots,f_k)\in\mathcal{F}(G_1)\times\cdots\times \mathcal{F}(G_k)$ be functions that satisfy ${\cal S}_j$-polarity condition with respect to some $\rho$. In order to prove the desired inequality \eqref{eq-prop-equiv-2'}, (by an approximation argument) we can assume that $\lim_{t\to\infty}\rho(t)=0$, $\rho$ is continuous, strictly decreasing and that $\lim_{t\to 0^+}\rho(t)=\infty$. Define the (not necessarily convex) sets $K_i(r_i):=\{x_i\in\R^n:f_i(x_i)\geq r_i\}$, $r_i\geq 0$ and notice that $K_i(r_i)\in\mathcal{S}(G_i)$, $i=1,\ldots,k$. From ${\cal S}_j$-polarity condition one obtains that, for $x_i\in K_i(r_i)$, $i=1,\ldots,k$, it holds
$$r_1\ldots r_k\leq \prod_{i=1}^kf_i(x_i)\leq \rho\left({\cal S}_j(x_1,\ldots,x_k)\right).$$
Moreover, using the the strict monotonicity of $\rho$, we get
$${\cal S}_j(x_1,\ldots,x_k)\leq \rho^{-1}(r_1\ldots r_k).$$
Consequently, by the fact that ${\cal S}_j$ is homogeneous of order $j$, setting $\lambda:={{k}\choose{j}}^{\frac{1}{j}}\rho^{-1}(r_1\cdots r_k)^{-\frac{1}{j}},$ we conclude
$${\cal S}_j(\lambda x_1,\ldots,\lambda x_k)\leq {{k}\choose{j}}.$$Thus, by the assumption that (i) holds true and  by Lemma \ref{kjsdncs} we obtain
\begin{equation}\label{sdcskdcs}
\mu(\lambda K_1(r_1))\ldots \mu(\lambda K_k(r_k))\leq \mu(\textnormal{conv}(\lambda K_1(r_1)))\ldots \mu(\textnormal{conv}(\lambda K_k(r_k)))\leq \mu(B^n_j)^k.
\end{equation}
Equivalently, using the homogeneity of $\mu$, one has
$$\mu(K_1(r_1))\ldots \mu(K_k(r_k))\leq \frac{\mu(B^n_j)^k}{\lambda^{ka}}={{k}\choose{j}}^{-\frac{ka}{j}}\mu(B^n_j)^k\rho^{-1}(r_1\cdots r_k)^{\frac{ka}{j}}.$$
Set  $\phi_i(r_i):=\mu(K_i(r_i))$ , $r_i\geq 0$, $i=1,\ldots,k$ and $\phi(r):={{k}\choose{j}}^{-\frac{a}{j}}\mu(B^n_j)\rho^{-1}(r^k)^{\frac{a}{j}}$, $r\geq 0$. Then, the previous inequality can be written as
$$(\varphi_1(r_1)\ldots\varphi_k(r_k))^{1/k}\leq \varphi\left((r_1\ldots r_k)^{1/k}\right)$$
and, therefore, the Prekopa-Leindler inequality (Theorem \ref{Prekopa-Leindlre}) together with the Layer-Cake formula give
\begin{equation}\label{eq-prop-C2->C1-1}
\prod_{i=1}^k\int_{\R^n} f_i(x_i)\,d\mu(x_i)=\prod_{i=1}^k\int_0^{\infty} \varphi_i(r_i)\,dr_i\leq\left(\int_0^{\infty} \varphi\right)^k={{k}\choose{j}}^{-\frac{ka}{j}}\mu(B^n_j)^k\left(\int_0^{\infty}\rho^{-1}(r^k)^{\frac{a}{j}}\, dr\right)^k.
\end{equation}
On the other hand, using the extra assumptions on $\rho$ and the homogeneity of $\mu$, we see that
\begin{eqnarray}\label{eq-prop-C2->C1-3}
\int_{\mathbb{R}^n}\rho\left({{k}\choose{j}}\|u\|^j_j\right)^{1/k}\, d\mu(u)&=&\int_0^{\infty}\mu\left(\Big\{u:\rho\left({{k}\choose{j}}\|u\|_j^j\right)\geq t^k\Big\}\right)\,dt\nonumber\\
&=&\int_0^{\infty}\mu\left(\Big\{ u: \|u\|_j\leq \left({{k}\choose{j}}^{-1} \rho^{-1}(t^k)\right)^{\frac{1}{j}}\Big\}\right)\,dt\nonumber\\
&=&{{k}\choose{j}}^{-\frac{
a}{j}}\mu(B^n_j)\int_0^{\infty}\rho^{-1}(t^k)^{\frac{a}{j}}\,dt.
\end{eqnarray}
Putting together \eqref{eq-prop-C2->C1-1} and \eqref{eq-prop-C2->C1-3}, we arrive at \eqref{eq-prop-equiv-2'},
as claimed.
\end{proof}
The proof of Theorem \ref{KWconj} follows immediately from Theorem \ref{phisantalo} and Proposition \ref{vjndsf}.

\begin{remark}\label{boundremark2}
	Let $a_{n,j,k}$ be the positive constant defined in Remark \ref{boundremark}. Replacing $\mu(B^n_j)^k$ in \eqref{sdcskdcs} by $a_{n,j,k}\mu(B^n_j)^k$ and taking $\mu$ to be the Lebesgue measure, the proof of Proposition \ref{vjndsf} gives 
	\begin{equation*}
	\prod_{i=1}^k\int_{\R^n}f_i(x_i)\, dx_i\leq a_{n,j,k}\left(\int_{\R^n}\rho\left({{k}\choose{j}}\|u\|_j^j\right)^{1/k}\, du\right)^k,
	\end{equation*}
	for any even $f_1,\ldots,f_k$ satisfying $\mathcal{S}_j$-polarity condition with respect to some decreasing function $\rho$.
\end{remark}

\section{Ball's functional for many sets}\label{multi-Ball-section}
\hspace*{1.5em}Let us recall the definition of Ball's functional, mentioned in the Introduction. If $K$ is a symmetric convex body in $\R^n$, $B(K)$ is given by
$$B(K):=\int_{K}\int_{K^o}\langle x,y \rangle^2\, dx\, dy.$$  
It can be easily checked that $B(\cdot)$ is invariant under non-singular linear maps.
The primary goal of this section is to state and discuss a natural (at least in our opinion) extension of Conjecture \ref{Ball-conjec}, to the multi-entry setting. Let $\mathcal{D}(n)$ be the set of all orthonormal basis' in $\R^n$. For $k\geq 2$, $j\in\{2,\ldots,k\}$ and $\{\epsilon_m\}\in\mathcal{D}(n)$, define
\begin{equation*}
\mathcal{B}_j(K_1,\ldots,K_k,\{\epsilon_m\}):=\sum_{m=1}^n\prod_{i=1}^k\int_{K_i}|\langle x_i,\epsilon_m\rangle|^j\,dx_i.
\end{equation*}
Define, also
\begin{equation*}\label{def-B}
\mathcal{B}_j(K_1,\ldots,K_k):=\min_{\{\epsilon_m\}\in\mathcal{D}(n)}\mathcal{B}_j(K_1,\ldots,K_k,\{\epsilon_m\}).
\end{equation*}
One might dare to conjecture the following. Let $K_1,\ldots,K_k$ be symmetric convex bodies in $\R^n$ satisfying $\mathcal{E}_j$-polarity condition, $j\geq 2$.
Then,
\begin{equation}\label{multiBALLineq111}\mathcal{B}_j(K_1,\ldots,K_k)\leq \mathcal{B}_j(B^n_j,\ldots,B^n_j).
\end{equation}
It should be mentioned (this follows immediately from Lemma \ref{l-after-prop-connection} below) that 
\begin{equation}\label{eq-er-1}
\mathcal{B}_j(B^n_j,\ldots,B^n_j)=\mathcal{B}_j(B^n_j,\ldots,B^n_j,\{e_m\})=n\left(\int_{B^n_j}|\langle x,e_1\rangle|^j\, dx\right)^k.\end{equation}

Let us demonstrate that the conjectured inequality \eqref{multiBALLineq111} for $k=j=2$ agrees with Ball's Conjecture \ref{Ball-conjec}. First it will be useful to recall the notion of isotropicity. A symmetric convex body $K$ in $\R^n$ is called isotropic if 
$$\int_K\langle x,u\rangle^2\, dx=\frac{\|u\|_2^2}{n}\int_K\|x\|_2^2\, dx,\qquad \forall u\in\R^n.$$
Notice (see \cite{Mi-Pa}) that there is always a linear image $TK$ of $K$, such that $TK$ is isotropic.
Next, assume that conjecture \ref{Ball-conjec} is true. Observe that there always exists an orthonormal basis $\{\epsilon_m\}$ such that, for $i\neq m$, it holds
\begin{equation*}
\int_{K_1}\langle x,\epsilon_i\rangle\langle x,\epsilon_m\rangle\,dx=0.
\end{equation*}
Hence,
\begin{eqnarray*}
\mathcal{B}_2(K_1,K_2)\leq\mathcal{B}_2(K_1,K_2,\{\epsilon_m\})
&\leq&\mathcal{B}_2(K_1,K_1^o,\{\epsilon_m\})\\
&=&\sum_{m=1}^n\int_{K_1}\langle x,\epsilon_m\rangle^2\,dx\int_{K_1^o}\langle y,\epsilon_m\rangle^2\,dy\\
&=&\int_{K_1}\int_{K_1^o}\langle x,y\rangle^2\,dx\,dy\\
&\leq&\int_{B^n_2}\int_{B^n_2}\langle x,y\rangle^2\,dx\,dy=\mathcal{B}_2(B^n_2,B^n_2).
\end{eqnarray*}
Conversely, assume that \eqref{multiBALLineq111} is true for $k=j=2$ and for all symmetric convex bodies $K_1, K_2$.  One can take $K_1=K=K_2^o$. Since $B(K)$ is invariant under non-singular linear maps, we can assume that $K$ is isotropic. We have
\begin{eqnarray*}
B(B^n_2)=\mathcal{B}_2(B^n_2,B^n_2)\geq\mathcal{B}_2(K,K^o)&=&\min_{\{\epsilon_m\}\in\mathcal{D}(n)}\sum_{m=1}^n\int_{K}\langle x,\epsilon_m\rangle^2\,dx\int_{K^o}\langle y,\epsilon_m\rangle^2\,dy\\
&=&\min_{\{\epsilon_m\}\in\mathcal{D}(n)}\sum_{m=1}^n\int_{K}\langle x,\epsilon_1\rangle^2\,dx\int_{K^o}\langle y,\epsilon_m\rangle^2\,dy\\
&=&\int_K\langle x,\epsilon_1\rangle^2\,dx\int_{K^o}\|y\|_2^2\,dx\\
&=&\frac{1}{n}\int_K\|x\|_2^2\,dx\int_{K^o}\|y\|_2^2\,dy\\
&=&\int_{K}\int_{K^o}\langle x,y\rangle^2\,dx\,dy=B(K).
\end{eqnarray*}

Next, we would like to explain the connection between the conjectured inequality \eqref{multiBALLineq111} and the $j$-Santal\'o conjecture \ref{S_j Santalo for bodies}.
\begin{proposition}\label{jfvnbdf}
Let $k\geq 2$ be a positive integer, $j\in\{2,\ldots,k\}$ and $K_1,\ldots,K_k$ be symmetric convex bodies satisfying $\mathcal{E}_j$-polarity condition. If \eqref{multiBALLineq111} holds, then \eqref{santalo-for-many} also holds.
\end{proposition}
Proposition \ref{jfvnbdf} follows immediately from the following lemma (the corresponding fact involving $B(\cdot)$ was obtained by Ball \cite{Ball-f-Santalo} \cite{Ball-conjecture}; see also Lutwak \cite{Lu}).
\begin{lemma}\label{l-after-prop-connection}
For convex bodies $K_i$, $i=1,\ldots,k$ we have
\begin{equation}\label{jnjndsf}
\frac{{\cal B}_j(B_j^n,\dots,B_j^n,\{e_m\})}{|B_j^n|^{\frac{k(n+j)}{n}}}\leq \frac{\mathcal{B}_j(K_1,\ldots,K_k)}{(|K_1|\cdots|K_k|)^{\frac{n+j}{n}}}.
\end{equation}
\end{lemma}
\begin{proof}
We may assume that $$\mathcal{B}_j(K_1,\ldots,K_k,\{e_m\})=\mathcal{B}_j(K_1,\ldots,K_k).$$
Let $Q$ be a convex body in $\R^n$. 
We will need the following simple fact.

\textit{Fact.} Let $T\in SL(n)$ be a diagonal positive definite map (with respect to the basis $\{e_m\}$). Then,
$$\prod_{m=1}^n\int_{TQ}|\langle x,e_m\rangle|^j\, dx=\prod_{m=1}^n\int_Q|\langle x,e_m\rangle|^j\, dx.$$
Furthermore, there exists a diagonal positive definite map $T_0\in SL(n)$, such that 
$$\int_{T_0Q}|\langle x,e_1\rangle|^j\, dx=\ldots=\int_{T_0Q}|\langle x,e_n\rangle|^j\, dx.$$

It follows that
\begin{eqnarray*}
\left(\prod_{m=1}^n\int_Q|\langle x,e_m\rangle|^j\, dx\right)^{1/n}&=&\left(\prod_{m=1}^n\int_{T_0Q}|\langle x,e_m\rangle|^j\, dx\right)^{1/n}\\
&=&\frac{1}{n}\sum_{m=1}^n\int_{T_0Q}|\langle x,e_m\rangle|^j\, dx\\
&=&\frac{1}{n}\int_{T_0Q}\|x\|_j^j\, dx\\
&=&\frac{1}{n}\int_0^\infty|(T_0Q)\cap \{x:\|x\|_j^j\geq t\}|\, dt\\
&=&\frac{1}{n}\int_0^\infty\left(|T_0Q|-|(T_0Q)\cap\{x:\|x\|_j<t^{1/j}\}|\right)\, dt\\
&=&\frac{1}{n}\int_0^\infty\left(|Q|-|(T_0Q)\cap (t^{1/j}B_j^n)|\right)\, dt.
\end{eqnarray*}
Since, for all $t>0$, it holds
$$|(T_0Q)\cap (t^{1/j}B_j^n)|\leq \Big|\left(\left(|T_0Q|/|B_j^n|\right)^{1/n}B_j^n\right)\cap \left(t^{1/j}B_j^n\right)\Big|=\Big|\left(\left(|Q|/|B_j^n|\right)^{1/n}B_j^n\right)\cap \left(t^{1/j}B_j^n\right)\Big|,$$
we arrive at
\begin{eqnarray}\label{eq-ineq-final-lemma}
\left(\prod_{m=1}^n\int_Q|\langle x,e_m\rangle|^j\, dx\right)^{1/n}&\geq&\frac{1}{n}\int_{\left(|Q|/|B_j^n|\right)^{1/n}B_j^n}\|x\|_j^j\, dx\nonumber\\
&=&\frac{1}{n}\left(\frac{|Q|}{|B_j^n|}\right)^{\frac{n+j}{n}}\int_{B_j^n}\|x\|_j^j\, dx=:c_{n,j}|Q|^{\frac{n+j}{n}},
\end{eqnarray}
where $c_{n,j}$ is a positive constant that depends only on $n$ and $j$, such that equality holds in \eqref{eq-ineq-final-lemma} if $Q=B_j^n$.

By the arithmetic-geometric mean inequality and \eqref{eq-ineq-final-lemma}, one has
\begin{eqnarray*}
\mathcal{B}_j(K_1,\ldots,K_k,\{e_m\})&\geq& n\prod_{m=1}^n\left(\prod_{i=1}^k\int_{K_i}|\langle x_i,e_m\rangle|^j\, dx\right)^{1/n}\\
&=&n\prod_{i=1}^k\left(\prod_{m=1}^n\int_{K_i}|\langle x_i,e_m\rangle|^j\, dx\right)^{1/n}\geq n(c_{n,j})^k\prod_{i=1}^k|K_i|^{\frac{n+j}{n}}.
\end{eqnarray*}
Notice that if $K_1=\ldots=K_k=B^n_j$, then equality holds in all previous inequalities. This finishes the proof of the Lemma.
\end{proof}
Finally, we would like to extend the definition of the ${\cal B}_j$ functional,  to tuples of functions instead of tuples of convex bodies. For even non-negative integrable functions $f_1,\ldots,f_k$, $2\leq j\leq k$ and $\{\epsilon_m\}\in\mathcal{D}(n)$, set
\begin{equation*}
\mathcal{B}_j(f_1,\ldots,f_k,\{\epsilon_m\}):=\sum_{m=1}^n\prod_{i=1}^k\int_{\R^n}|\langle x_i,\epsilon_m\rangle|^j f_i(x_i)\,dx_i
\end{equation*}
and
\begin{equation*}
\mathcal{B}_j(f_1,\ldots,f_k):=\min_{\{\epsilon_m\}\in\mathcal{D}(n)}\mathcal{B}_j(f_1,\ldots,f_k,\{\epsilon_m\}).
\end{equation*}
The functional version of the conjectured inequality \eqref{multiBALLineq111} states the following. Let $f_i:\R^n\to\R_+$, $i=1,\ldots,k$, be even functions satisfying ${\cal S}_j$-polarity condition with respect to some non-negative and decreasing function $\rho$. Then,
\begin{equation}\label{multiBALLineq1}
\mathcal{B}_j(f_1,\ldots,f_k)\leq n\left(\int_{\R^n}|\langle u,e_1\rangle|^j\rho\left({{k}\choose{j}}\|u\|_j^j\right)^{\frac{1}{k}}\,du\right)^k=n^{1-k}\left(\int_{\R^n}\|u\|_j^j\rho\left({{k}\choose{j}}\|u\|_j^j\right)^{\frac{1}{k}}\,du\right)^k.
\end{equation}
By \eqref{observation}, \eqref {multiBALLineq1} would immediately imply \eqref{multiBALLineq111}.
Using \eqref{observation}, Proposition \ref{jfvnbdf},  and Proposition \ref{vjndsf} we obtain the following.
\begin{corollary}
If \eqref{multiBALLineq1} holds for any even non-negative integrable functions $f_1,\ldots,f_k$ and any non-negative decreasing function $\rho$, then the functional $j$-Santal\'{o} conjecture \ref{jvnljfnv} holds in full generality.
\end{corollary}
Furthermore, notice that Proposition \ref{Sjp-BALL-orthants} and \eqref{observation} imply the following.
\begin{corollary}
Inequalities \eqref{multiBALLineq111} and \eqref{multiBALLineq1} are both true in the unconditional case.
\end{corollary}  

We mention that the authors in \cite{kinezoi} proved the following functional version of Ball's inequality: If $\rho:\R\to\R_+$ is a measurable function and $f_1,f_2:\R^n\to\R_+$ are integrable unconditional log-concave functions satisfying $f_1(x_1)f_2(x_2)\leq \rho(\langle x_1,x_2\rangle)$, for all $x_1,x_2\in\R^n$, then
\begin{equation}\label{eq-kinezoi}
\int_{\R^n}\int_{\R^n}\langle x,y\rangle^2f_1(x)f_2(y)\, dx\, dy\leq   n^{-1}\left(\int_{\R^n}\|u\|_2^2\rho\left(\|u\|_2^2\right)^{\frac{1}{2}}\,du\right)^2
\end{equation}
It is unknown if \eqref{eq-kinezoi} holds for arbitrary even log-concave functions. Using similar arguments as in the  case of sets, one can show that the conjectured inequality \eqref{multiBALLineq1} (for arbitrary even integrable functions) for $k=j=2$ is equivalent to \eqref{eq-kinezoi}. Hence, \eqref{multiBALLineq1} for unconditional functions can be interpreted as an extension of the functional version of Ball's inequality to the multi-entry setting, if $\rho$ is additionally assumed to be decreasing.
\\
\\
\textbf{Acknowledgment.} In a previous version of this manuscript, we mistakenly claimed that Conjectures \ref{S_j Santalo for bodies} and \ref{jvnljfnv} also hold for $j=1$. We would like to thank Matthieu Fradelizi for pointing us the error and for providing the example in Remark \ref{Fradelizi}. 

\vspace{1.5 cm}
\noindent 
Pavlos Kalantzopoulos \\
Department of Mathematics\\
Central European University\\
Budapest, Hungary, 1051 \\
E-mail address: kalantzopoulos\_pavlo@phd.ceu.edu \ \& \ paul-kala@hotmail.com\\
\noindent\\
Christos Saroglou \\
Department of Mathematics\\
University of Ioannina\\
Ioannina, Greece, 45110 \\
E-mail address: csaroglou@uoi.gr \ \& \ christos.saroglou@gmail.com
\end{document}